\documentclass[11pt]{article}
\usepackage{amsfonts}
\usepackage{amsmath}
\usepackage{amssymb}
\usepackage{geometry}

\newtheorem{theorem}{Theorem}

\newtheorem{lemma}{Lemma}

\newtheorem{remark}[theorem]{Remark}

\newenvironment{proof}[1][Proof]{\noindent\textbf{#1.} }{\ \rule{0.5em}{0.5em}}
\input{tcilatex}
\begin{document}

\begin{center}
Number theory{\small \\[0pt]
\rule{3cm}{0.5pt}}\\[0pt]
Hyperbolic Summation Involving the Function $\Omega \left( n\right) $ and $%
\func{lcm}$.\bigskip 

\textsf{Meselem KARRAS}

\smallskip

University of Tissemsilt.

FIMA Laboratory, Khemis Miliana, Algeria.

e-mail: \textsf{m.karras@univ-tissemsilt.dz}\bigskip \bigskip

\bigskip \smallskip
\end{center}

\textbf{Abstract}. We study the sum $\dsum\limits_{abc\leq x}\Omega \left( %
\left[ a,b,c\right] \right) ,$ where $\Omega (n)$ denotes the number of
distinct prime divisors of $n\in 
\mathbb{Z}
_{\geq 1}$ counted with multiplicity, and where $\left( a,b,c\right) =\gcd
\left( a,b,c\right) $ and $\left[ a,b,c\right] =\func{lcm}\left(
a,b,c\right) $. An asymptotic formula is derived for this sum over the
hyperbolic region $\left\{ \left( a,b,c\right) \in 
\mathbb{Z}
_{\geq 1}^{3},\text{ }abc\leq x\right\} .\bigskip $


\bigskip

\noindent \textbf{Introduction.}

Let $f$ be an artithmetic function,\ and for $r\geq 2,$ let $\left(
n_{1},n_{2},\cdots n_{r}\right) =\gcd \left( n_{1},n_{2},\cdots n_{r}\right) 
$ and $\left[ n_{1},n_{2},\cdots n_{r}\right] =\func{lcm}\left(
n_{1},n_{2},\cdots n_{r}\right) .$ Let $\omega \left( n\right) $ denote the
number of distinct prime divisors of a positive integer $n\geq 1$, and $%
\Omega \left( n\right) $ the number of prime divisors of $n$, counted with
multiplicity. The problem of finding asymptotic formula for sums such as%
\begin{equation*}
\sum_{n_{1}n_{2}\cdots n_{r}\leq x}f\left( \left( n_{1},n_{2},\cdots
n_{r}\right) \right) \text{ or }\sum_{n_{1}n_{2}\cdots n_{r}\leq x}f\left( %
\left[ n_{1},n_{2},\cdots n_{r}\right] \right)
\end{equation*}%
has been widely studied in number theory. Previous works by researchers such
as \cite{HEYMEN TOTH}. Results for the general case are often limited to
multiplicative arithmetic functions under certain conditions or to additive
arithmetic functions for the first type of sum.

In this work, we focus on the cases $f=\omega $ and $f=\omega ^{m},$ $m\geq
1 $ for the second type of sum. While the generalization to $f=\Omega $
poses significant challenges. This paper establishes an interesting result
for the case $r=3.$

We begin with two important results which are direct applications of two of
Ivic's theorems \cite{Ivic} .

\begin{theorem}
Let $r\geq 2$ be fixed integer and $N$ be an arbitrary fixed integer, but
for which $N>r$. Then there exist computable constants $a_{r,j}$, $b_{r,j},$ 
$c_{r,j}$ ($a_{r}$, $j\neq 0$) such that%
\begin{eqnarray}
\sum_{n_{1}n_{2}\cdots n_{r}\leq x}\omega \left( \left[ n_{1},n_{2},\cdots
n_{r}\right] \right) &=&x\sum_{j=1}^{N}\left( a_{r,j}\log \log
x+b_{r,j}\right) \log ^{r-j}x+x\sum_{j=r+1}^{N}c_{r,j}\log ^{r-j}x  \notag \\
&&+O\left( x\log ^{r-N-1}x\right) .  \label{7}
\end{eqnarray}
\end{theorem}

\begin{theorem}
Let $m,N\geq 1$ and $r\geq 2$ be fixed integers. Then there exist
polynomials $P_{r,m,j}\left( t\right) $ $(j=1,2,...,N)$ of degree $m$ in t
with computable coeffcients such that%
\begin{eqnarray}
\sum_{n_{1}n_{2}\cdots n_{r}\leq x}\omega ^{m}\left( \left[
n_{1},n_{2},\cdots n_{r}\right] \right) &=&x\sum_{j=1}^{N}P_{r,m,j}\left(
\log \log x\right) \log ^{r-j}x  \notag \\
&&+O\left( x\log ^{r-N-1}x\left( \log \log x\right) ^{m}\right) .  \label{6}
\end{eqnarray}
\end{theorem}

The authors in \cite[Theorem 2.9.]{hayman and Thoth} show the first result
in the case $r=2$ and the same result for the the function $\Omega \left(
n\right) $ (see Theorem $2.10$).

\begin{proof}
The distinct prime divisors of the integer $\left[ n_{1},n_{2},...n_{r}%
\right] $ are the same of the integer $n_{1}n_{2}...n_{r}.$ Then we obtain 
\begin{equation*}
\omega \left( \left[ n_{1},n_{2},...n_{r}\right] \right) =\omega \left(
n_{1}n_{2}...n_{r}\right) .
\end{equation*}%
Therefore, for any integer $m\geq 1$ 
\begin{eqnarray*}
\sum_{n_{1}n_{2}...n_{r}=n}\omega ^{m}\left( \left[ n_{1},n_{2},...n_{r}%
\right] \right) &=&\sum_{n_{1}n_{2}...n_{r}=n}\omega ^{m}\left( n\right) \\
&=&\omega ^{m}\left( n\right) \sum_{n_{1}n_{2}...n_{r}=n}1 \\
&=&\omega ^{m}\left( n\right) \tau _{r}\left( n\right) .
\end{eqnarray*}%
Thus 
\begin{eqnarray*}
\sum_{n_{1}n_{2}\cdots n_{r}\leq x}\omega ^{m}\left( \left[
n_{1},n_{2},\cdots n_{r}\right] \right) &=&\sum_{n\leq
x}\sum_{n_{1}n_{2}\cdots n_{r}=n}\omega ^{m}\left( n\right) \\
&=&\sum_{n\leq x}\omega ^{m}\left( n\right) \tau _{r}\left( n\right) .
\end{eqnarray*}%
So, our results are Theorems $1$ and $2$ in \cite{Ivic} .
\end{proof}

\begin{theorem}
We have 
\begin{equation*}
\sum_{abc\leq x}\Omega \left( \left[ a,b,c\right] \right) =\frac{3}{2}%
x\left( \log x\right) ^{2}\log \log x+3(b-1)x^{2}\log x+\frac{\left(
C_{2}-3b\right) }{2}x\left( \log x\right) ^{2}+O\left( x\log x\log \log
x\right) ,
\end{equation*}%
where $b=A+\dsum\limits_{p}\dfrac{1}{p\left( p-1\right) }$ such that$\
A=\gamma +\dsum\limits_{p}\left( \log \left( 1-\dfrac{1}{p}\right) +\dfrac{1%
}{p}\right) $ and%
\begin{equation*}
C_{2}=\dsum\limits_{p}\dfrac{1}{p^{3}-1}\thickapprox 0.1941.
\end{equation*}
\end{theorem}

The proof of the theorem is based on the following lemmas:

\begin{lemma}
Let $f$ be an arithmetic function. Then%
\begin{equation}
\sum_{abc\leq x}f\left( \left[ a,b,c\right] \right) =3\sum_{an\leq x}f\left(
a\right) \tau \left( n\right) -3x\sum_{n\leq x}\frac{1}{n}\sum_{ab=n}f\left(
\left( a,b\right) \right) +\sum_{abc\leq x}f\left( \left( a,b,c\right)
\right) +O\left( \sum_{ab\leq x}f\left( a,b\right) \right) .  \label{1}
\end{equation}
\end{lemma}

\begin{proof}
Using the inclusion-exclusion principle, we have%
\begin{eqnarray*}
\sum_{abc=n}f\left( \left[ a,b,c\right] \right) &=&\sum_{abc=n}f\left(
a\right) +\sum_{abc=n}f\left( b\right) +\sum_{abc=n}f\left( c\right)
-\sum_{abc=n}f\left( \left( a,b\right) \right) \\
&&-\sum_{abc=n}f\left( \left( a,c\right) \right) -\sum_{abc=n}f\left( \left(
b,c\right) \right) +\sum_{abc=n}f\left( \left( a,b,c\right) \right) \\
&=&3\sum_{abc=n}f\left( a\right) -3\sum_{abc=n}f\left( \left( a,b\right)
\right) +\sum_{abc=n}f\left( \left( a,b,c\right) \right) .
\end{eqnarray*}%
Thus, we obtain:%
\begin{eqnarray*}
\sum_{abc\leq x}f\left( \left[ a,b,c\right] \right) &=&3\sum_{abc\leq
x}f\left( a\right) -3\sum_{abc\leq x}f\left( \left( a,b\right) \right)
+\sum_{abc\leq x}f\left( \left( a,b,c\right) \right) \\
&=&3\sum_{an\leq x}f\left( a\right) \tau \left( n\right) -3\sum_{ab\leq
x}f\left( \left( a,b\right) \right) \sum_{c\leq \frac{x}{ab}}1+\sum_{abc\leq
x}f\left( \left( a,b,c\right) \right) .
\end{eqnarray*}%
Furthermore, since: 
\begin{eqnarray*}
\sum_{ab\leq x}f\left( \left( a,b\right) \right) \sum_{c\leq \frac{x}{ab}}1
&=&x\sum_{ab\leq x}\frac{f\left( \left( a,b\right) \right) }{ab}+O\left(
\sum_{ab\leq x}f\left( \left( a,b\right) \right) \right) \\
&=&x\sum_{n\leq x}\frac{1}{n}\sum_{ab=n}f\left( \left( a,b\right) \right)
+O\left( \sum_{ab\leq x}f\left( \left( a,b\right) \right) \right) .
\end{eqnarray*}%
So, considering this last formula, we obtain the desired result.
\end{proof}

\begin{lemma}
For $x\geq 2,$ we have%
\begin{equation}
\sum_{an\leq x}\Omega \left( a\right) \tau \left( n\right) =\frac{x}{2}%
\left( \log x\right) ^{2}\log \log x+(b-1)x^{2}\log x-\frac{b}{2}x\left(
\log x\right) ^{2}+O\left( x\log x\log \log x\right) ,  \label{5}
\end{equation}%
where $b=A+\dsum\limits_{p}\dfrac{1}{p\left( p-1\right) }$ such that$\
A=\gamma +\dsum\limits_{p}\left( \log \left( 1-\dfrac{1}{p}\right) +\dfrac{1%
}{p}\right) \approx 0.2614972....$
\end{lemma}

\begin{proof}
By the well-known estimate formula, 
\begin{equation*}
\sum_{n\leq x}\tau \left( n\right) =x\left( \log x+C\right) +O\left(
x^{\theta +\varepsilon }\right) ,\ 
\end{equation*}%
where, $C=2\gamma -1$ and $\frac{1}{4}<\theta <\frac{1}{2},$ we have%
\begin{eqnarray*}
\sum_{an\leq x}\Omega \left( a\right) \tau \left( n\right)
&=&\dsum\limits_{a\leq x}\Omega \left( a\right) \sum_{n\leq \frac{x}{a}}\tau
\left( n\right) \\
&=&\dsum\limits_{a\leq x}\Omega \left( a\right) \left( \frac{x}{a}\left(
\log \frac{x}{a}+C\right) +O\left( \left( \frac{x}{a}\right) ^{\theta
+\varepsilon }\right) \right) \\
&=&x\left( \log x+C\right) \dsum\limits_{n\leq x}\frac{\Omega \left(
n\right) }{n}-x\dsum\limits_{n\leq x}\frac{\Omega \left( n\right) \log n}{n}%
+O\left( x^{\theta +\varepsilon }\dsum\limits_{n\leq x}\frac{\Omega \left(
n\right) }{n^{\theta +\varepsilon }}\right) .
\end{eqnarray*}%
Estimate of the following sums $\dsum\limits_{n\leq x}\dfrac{\Omega \left(
n\right) }{n},$ $\dsum\limits_{n\leq x}\dfrac{\Omega \left( n\right) \log n}{%
n}$ and $\dsum\limits_{n\leq x}\dfrac{\Omega \left( n\right) }{n^{\theta
+\varepsilon }}.$ \newline
We know that, 
\begin{equation*}
\dsum\limits_{n\leq x}\Omega \left( n\right) =x\log \log x+bx+O\left( \frac{x%
}{\log x}\right) ,
\end{equation*}%
where $b$\ is given by $\left( \ref{5}\right) .$ We us a partial summation,
we get%
\begin{equation*}
\dsum\limits_{n\leq x}\frac{\Omega \left( n\right) }{n}=\left( \log x\right)
\log \log x+\left( b-1\right) x+O\left( \log \log x\right) .
\end{equation*}
Again by partial summation, we have%
\begin{eqnarray*}
\dsum\limits_{n\leq x}\frac{\Omega \left( n\right) \log n}{n} &=&\log
x\dsum\limits_{n\leq x}\frac{\Omega \left( n\right) }{n}-\int_{2}^{x}\frac{1%
}{t}\left( \dsum\limits_{n\leq t}\frac{\Omega \left( n\right) }{n}\right) dt
\\
&=&\frac{1}{2}\left( \log x\right) ^{2}\log \log x+\frac{b}{2}\left( \log
x\right) ^{2}+O\left( \left( \log x\right) \log \log x\right) ,
\end{eqnarray*}%
and we have 
\begin{equation*}
\dsum\limits_{n\leq x}\frac{\Omega \left( n\right) }{n^{\theta +\varepsilon }%
}=O\left( x^{1-\theta -\varepsilon }\log \log x\right) .
\end{equation*}
\end{proof}

\begin{remark}
We can use the Theorem $3$ from \cite{Ivic} with parameters $k=3,$ $m=0$ and 
$r=1.$
\end{remark}

\begin{lemma}
We have 
\begin{equation}
\sum_{n\leq x}\frac{1}{n}\sum_{ab=n}\Omega \left( \left( a,b\right) \right) =%
\frac{C_{\Omega }}{2}\log ^{2}\left( x\right) +\left( C_{\Omega }+1\right)
\log x+D_{\Omega }+O\left( x^{-1/2}\right) ,  \label{4}
\end{equation}%
where $C_{\Omega }$ and $D_{\Omega }$ are two positive constants.
\end{lemma}

\begin{proof}
We apply a partial summation to the estimate $\left( 2.16\right) $ in \cite%
{Bord et Toth}$,$ we get $\left( \ref{4}\right) $
\end{proof}

\begin{lemma}
We have 
\begin{equation}
\sum_{abc\leq x}\Omega \left( \left( a,b,c\right) \right) =\frac{C_{2}}{2}%
x\log ^{2}\left( x\right) +O\left( x\log x\right) ,  \label{12}
\end{equation}%
where $C_{2}=\dsum\limits_{p}\dfrac{1}{p^{3}-1}\thickapprox 0.1941.$
\end{lemma}

We note that formula $\left( \ref{12}\right) $ has been proved in the
general case in \cite[Theorem 2.3]{HEYMEN TOTH} , So here we give the
explicit form of the polynomial $P_{\Omega ,2}\left( x\right) $ of degree $2$%
. First, by lemma $5.1$ in \cite{Kartzal et N et Thoth}, we get 
\begin{equation*}
\sum_{abc=n}\Omega \left( \left( a,b,c\right) \right) =\sum_{d^{3}m=n}\left(
\mu \ast \Omega \right) \left( d\right) \tau _{3}\left( m\right) ,
\end{equation*}%
then%
\begin{eqnarray*}
\sum_{abc\leq x}\Omega \left( \left( a,b,c\right) \right)
&=&\sum_{d^{3}m\leq x}\left( \mu \ast \Omega \right) \left( d\right) \tau
_{3}\left( m\right) \\
&=&\sum_{d\leq x^{1/3}}\left( \mu \ast \Omega \right) \left( d\right)
\sum_{m\leq \left( \frac{x}{d}\right) ^{1/3}}\tau _{3}\left( m\right) .
\end{eqnarray*}%
We use the estimate%
\begin{equation*}
\sum_{m\leq x}\tau _{3}\left( m\right) =\frac{x\log ^{2}x}{2}+O\left( x\log
x\right) ,
\end{equation*}%
see, e.g., Nathanson \cite[Th 7.6]{M. B. Nathanson}. According to this
estimate%
\begin{eqnarray}
\sum_{abc\leq x}\Omega \left( \left( a,b,c\right) \right) &=&\sum_{d\leq
x^{1/3}}\left( \mu \ast \Omega \right) \left( d\right) \left( \frac{x}{2d^{3}%
}\log ^{2}\left( \frac{x}{d^{3}}\right) +O\left( \frac{x}{d^{3}}\log \left( 
\frac{x}{d^{3}}\right) \right) \right)  \notag \\
&=&\frac{x\log ^{2}x}{2}\sum_{d\leq x^{1/3}}\frac{\left( \mu \ast \Omega
\right) \left( d\right) }{d^{3}}+\left( -3x\log x+\frac{9}{2}x\right)
\sum_{d\leq x^{1/3}}\frac{\left( \mu \ast \Omega \right) \left( d\right)
\log d}{d^{3}}  \notag \\
&&+O\left( x\log x\sum_{d\leq x^{1/3}}\frac{\left( \mu \ast \Omega \right)
\left( d\right) }{d^{3}}\right) .  \label{14}
\end{eqnarray}%
On the other hand, we have 
\begin{eqnarray*}
\sum_{n=1}^{\infty }\frac{\left( \mu \ast \Omega \right) \left( n\right) }{%
n^{s}} &=&\sum_{n=1}^{\infty }\frac{\mu \left( n\right) }{n^{s}}\times
\sum_{n=1}^{\infty }\frac{\Omega \left( n\right) }{n^{s}} \\
&=&\frac{1}{\zeta \left( s\right) }\times \zeta \left( s\right) \sum_{p}%
\frac{1}{p^{s}-1} \\
&=&\sum_{p}\frac{1}{p^{s}-1},\text{ }\func{Re}\left( s\right) >1.
\end{eqnarray*}%
then 
\begin{eqnarray}
\sum_{d\leq x^{1/3}}\frac{\left( \mu \ast \Omega \right) \left( d\right) }{%
d^{3}} &=&\sum_{d=1}^{\infty }\frac{\left( \mu \ast \Omega \right) \left(
d\right) }{d^{3}}+O\left( \frac{1}{x^{2/3}}\right)  \notag \\
&=&\sum_{p}\frac{1}{p^{3}-1}+O\left( \frac{1}{x^{2/3}}\right) =C_{2}+O\left( 
\frac{1}{x^{2/3}}\right) .  \label{13}
\end{eqnarray}%
Using this last estimate with a partial sum, we find%
\begin{equation}
\sum_{d\leq x^{1/3}}\frac{\left( \mu \ast \Omega \right) \left( d\right)
\log d}{d^{3}}=O\left( 1\right)  \label{10}
\end{equation}%
According $\left( \ref{13}\right) $ and $\left( \ref{10}\right) $ in $\left( %
\ref{14}\right) $, therefore we obtain $\left( \ref{12}\right) .$

\begin{proof}[Proof of Thorem 3]
By lemme $1$, we have%
\begin{equation*}
\sum_{abc\leq x}\Omega \left( \left[ a,b,c\right] \right) =3\sum_{an\leq
x}\Omega \left( a\right) \tau \left( n\right) -3x\sum_{n\leq x}\frac{1}{n}%
\sum_{ab=n}\Omega \left( \left( a,b\right) \right) +\sum_{abc\leq x}\Omega
\left( \left( a,b,c\right) \right) +O\left( \sum_{ab\leq x}\Omega \left(
a,b\right) \right) ,
\end{equation*}%
and by lemme $2,$ $3$ and $4$ we get 
\begin{equation*}
\sum_{abc\leq x}\Omega \left( \left[ a,b,c\right] \right) =\frac{3}{2}%
x\left( \log x\right) ^{2}\log \log x+3(b-1)x^{2}\log x+\frac{\left(
C_{2}-3b\right) }{2}x\left( \log x\right) ^{2}+O\left( x\log x\log \log
x\right) ,
\end{equation*}%
which concludes the proof.\bigskip
\end{proof}


\begin{thebibliography}{9}
\bibitem{Bord et Toth} O. Bordell\`{e}s and L. T\'{o}th. Additive arithmetic
functions meet the inclusion-exclusion principle.

\bibitem{hayman and Thoth} Heyman, R., T\'{o}th, L. On Certain Sums of
Arithmetic Functions Involving the GCD and LCM of Two Positive Integers.
Results Math $76$, $49$ $(2021)$.

\bibitem{HEYMEN TOTH} Heyman, R., T\'{o}th, L. Hyperbolic summation for
functions of the GCD and LCM of several integers, The Ramanujan Journal $%
62(1)$, $1-18$ $\left( 2022\right) .$

\bibitem{A. Iksanov} A. Iksanov, A. Marynych, and K. Raschel, Asymptotics of
arithmetic functions of GCD and LCM of random integers in hyperbolic
regions, Preprint, $2021$, arXiv: 2112.11892v1 [math.NT].

\bibitem{Ivic} A. Ivic, \ Sums of products of certain arthmetical functions,
Publications de l'institut math\'{e}matique (Beograd) (N.S.) Vol. $41(55)$,
pp. $31-41(1987).$

\bibitem{Kartzal et N et Thoth} E. Kr\"{a}tzel, W. G. Nowak and L. T\'{o}th,
On certain arithmetic functions inolving the greatest common divisor, Cent.
Eur. J. Math. $10(2012)$, $761-774$.

\bibitem{M. B. Nathanson} M. B. Nathanson, Elementary Methods in Number
Theory, Graduate Texts in Mathematics,$195$, Springer, $2000$.
\end{thebibliography}
\end{document}